\documentclass{svmult}

\renewcommand{\email}[1]{\emailname: #1} 

\usepackage{mathptmx}       
\usepackage{helvet}         
\usepackage{courier}        

\usepackage{makeidx}         
\usepackage{graphicx}        
\usepackage[bottom]{footmisc}

\usepackage{latexsym}
\usepackage{amsmath}
\usepackage{amsfonts}
\usepackage{amssymb}
\usepackage{bm}

\usepackage{url}
\usepackage{algorithm}
\usepackage{algorithmic}
\usepackage[misc,geometry]{ifsym}

\renewenvironment{proof}{\noindent{\itshape Proof.}}{\smartqed\qed}

\spdefaulttheorem{assumption}{Assumption}{\upshape \bfseries}{\itshape}
\spdefaulttheorem{algo}{Algorithm}{\upshape \bfseries}{\itshape}

\newcommand{\bsc}{{\boldsymbol{c}}}

\newcommand{\bsu}{{\boldsymbol{u}}}

\newcommand{\bsx}{{\boldsymbol{x}}}

\newcommand{\bsz}{{\boldsymbol{z}}}

\newcommand{\rd}{{\mathrm{d}}}

\newcommand{\bbE}{{\mathbb{E}}}

\DeclareSymbolFont{bbold}{U}{bbold}{m}{n}
\DeclareSymbolFontAlphabet{\mathbbold}{bbold}

\begin{document}

\title*{Quasi-Monte Carlo for an Integrand with a Singularity along a Diagonal in the Square}

\author{Kinjal Basu \and Art B. Owen}

\institute{
 Kinjal Basu
 \at LinkedIn Inc.\\
 \email{kbasu@linkedin.com}
 \and
 Art B.\ Owen
 \at Stanford University\\
 \email{owen@stanford.edu}
}

\maketitle

\index{Basu, K.}
\index{Owen, A. B.}  

\abstract{
Quasi-Monte Carlo methods are designed for integrands of bounded 
variation, and this excludes singular integrands.  Several methods 
are known for integrands that become singular on the boundary 
of the unit cube $[0,1]^d$ or at isolated possibly unknown points 
within $[0,1]^d$.  Here we consider functions on the 
square $[0,1]^2$ that may become singular as the point approaches 
the diagonal line $x_1=x_2$, and we study three quadrature
methods.
The first method splits the square into two triangles separated
by a region around the line of singularity, and applies
recently developed triangle QMC rules to the two triangular parts.
For functions with a singularity `no worse than $|x_1-x_2|^{-A}$ is'
for $0<A<1$ that method yields an error of 
$O( (\log(n)/n)^{(1-A)/2})$.  
We also consider methods extending the integrand into
a region containing the singularity and show that method
will not improve upon using two triangles. Finally, we consider
transforming the integrand to have a more QMC-friendly
singularity along the boundary of the square. This then
leads to error rates of $O(n^{-1+\epsilon+A})$ when combined with
some corner-avoiding Halton points or with randomized QMC
but it requires some stronger assumptions on the original singular integrand.
}

\section{Introduction}

Quasi-Monte Carlo (QMC)  integration is designed for integrands of bounded 
variation in the sense of Hardy and Krause (BVHK). 
Such integrands must necessarily be bounded.  Singular integrands 
cannot be BVHK; they cannot even be Riemann integrable. 
It is known 
since \cite{debr:post:1968} and \cite{bind:1970} that for any integrand  $f$ on $[0,1]^d$ that is not 
Riemann integrable, there exists a sequence $\bsx_i\in[0,1]^d$ for which 
the star discrepancy $D_n^*(\bsx_1,\dots,\bsx_n)\to 0$
as $n\to\infty$ while 
$(1/n)\sum_{i=1}^nf(\bsx_i)$ fails to converge to $\int_{[0,1]^d}f(\bsx)\rd\bsx$. 

We are interested in problems where the singularity arises
along a manifold in $[0,1]^d$. For motivation, see the engineering
applications by Mishra and Gupta in \cite{mishra2009application}
and several other papers.  Apart from a few remarks, we focus solely on the problem where there is
a singularity along the line $x_1=x_2$ in $[0,1]^2$.

It is possible for QMC integration to succeed on unbounded integrands. 
Sobol' \cite{sobo:1973} noticed this when colleagues used his methods on such 
problems. He explained it in terms of QMC points that avoid 
a hyperbolic region around the lower boundary of the unit cube 
where the integrands became singular. 
Klinger \cite{klin:1997} shows that Halton points and some digital nets 
avoid a cubical region around the origin. 
Halton points (after the zero'th)  avoid hyperbolic 
regions around the boundary faces of the unit  cube at a rate suitable to get error bounds for QMC
\cite{haltavoid}. 
Certain Kronecker sequences avoid hyperbolic 
regions around the boundary of the cube \cite{klin:1997:b}. 
In all of these examples, avoiding the singularity should be understood as using points that approach it, but not too quickly, as the number $n$ of function evaluations increases. 

For plain Monte Carlo, the location of the singularity is not important. 
One only needs to consider the first two moments of the integrand. 
Because QMC exploits mild smoothness of the integrand, the nature 
of the singularity matters.  
Reference \cite{owen2006quasi} considers randomized 
QMC (RQMC) methods for integrands with point singularities at unknown locations.
In RQMC, the integrand is evaluated at points that, individually, are uniformly 
distributed on $[0,1]^d$ and this already implies a singularity avoidance 
property via the Borel-Cantelli lemma. 
If $\int f(\bsx)^2\rd\bsx<\infty$ then 
scrambled nets yield an 
unbiased estimate of $\mu=\int f(\bsx)\rd\bsx$ with RMSE $o(n^{-1/2})$
\cite{owen96}.

The analyses in \cite{haltavoid} and \cite{owen2006quasi} 
employ an extension $\tilde f$  of $f$ from  a set
$K=K_n\subset[0,1]^d$ to $[0,1]^d$.  The extension satisfies 
$\tilde f(\bsx) =f(\bsx)$ for $\bsx\in K$. 
Now the quadrature error is 
\begin{align*}
\frac1n\sum_{i=1}^nf(\bsx_i) -\int_{[0,1]^d} f(\bsx) \rd\bsx
& = \frac1n\sum_{i=1}^nf(\bsx_i) -\frac1n\sum_{i=1}^n\tilde f(\bsx_i) \\
&\phantom{=} + \frac1n\sum_{i=1}^n\tilde f(\bsx_i) 
-\int_{[0,1]^d} \tilde f(\bsx)\rd\bsx\\
&\phantom{=} + 
\int_{[0,1]^d} \tilde f(\bsx)\rd\bsx 
-\int_{[0,1]^d} f(\bsx)\rd\bsx. 
\end{align*}
If all of the points satisfy $\bsx_i\in K$, then 
the first term drops out and we find that 
\begin{align*}
\biggl|\frac1n\sum_{i=1}^nf(\bsx_i) -\int_{[0,1]^d} f(\bsx)\rd\bsx\biggr|
\le 
\biggl|\frac1n\sum_{i=1}^n \tilde f(\bsx_i) -\int_{[0,1]^d} \tilde f(\bsx) \rd\bsx\biggr|
+ \int_{-K}|\tilde f(\bsx)-f(\bsx)|\rd\bsx,
\end{align*}
where $-K=[0,1]^d\setminus K$. 
The extension used in \cite{haltavoid} and \cite{owen2006quasi} 
is due to Sobol' \cite{sobo:1973}.
It is particularly well suited to a Koksma-Hlawka 
bound for the first term above as $\tilde f$ has low variation.

In our case, we can isolate the singularity 
in the set $\{\bsx\mid |x_1-x_2|<\epsilon\}$. 
A set $K\subset[0,1]^d$ is Sobol'-extensible to $[0,1]^d$
with anchor $\bsc$ if for every $\bsx\in K$ the rectangle 
$\prod_{j=1}^d[\min(x_j,c_j),\max(x_j,c_j)]\subset K$. 
In our case, the set 
$\{\bsx\mid |x_1-x_2|\ge\epsilon\}$ in which $f$ is bounded 
is not Sobol' extensible. The extension $\tilde f$
used in \cite{haltavoid} and \cite{owen2006quasi} 
cannot be defined for this problem.

Section~\ref{sec:back} presents a strategy of avoiding 
a region near the singularity and integrating over two 
triangular regions using the method from \cite{basu:owen:2015}. 
The error is then a sum of two quadrature errors and one 
truncation error. We consider functions where the singularity 
is not more severe than that in $|x_1-x_2|^{-A}$ where 
$0<A<1$. 
Section~\ref{sec:bounds} shows that the truncation error 
in this approach is $O(\epsilon^{-A})$ and the quadrature error 
is $O( \epsilon^{-A-1}\log(n)/n)$
using the points from~\cite{basu:owen:2015} and 
a Koksma-Hlawka bound from \cite{bran:colz:giga:trav:2013}. 
The result is that we can attain a much better quadrature error 
bound of $O( (\log(n)/n)^{(1-A)/2})$. 
Section~\ref{sec:alt} shows that an approach 
based on finding an extension $\tilde f$ of $f$ would not 
yield a better rate for this problem. 
Section~\ref{sec:trans} transforms the problem so that each
triangular region becomes the image of a unit square, with
the singularity now on the boundary of the square.  
The singularity may be too severe for QMC. However, with
an additional assumption on the nature of the singularity
it is possible to attain a quadrature error of $O(n^{-1+\epsilon+A})$.
Section~\ref{sec:disc} summarizes the findings
and relates them to QMC-friendliness as discussed by
several authors, including Ian Sloan in his work with Xiaoqun Wang.

\section{Background}\label{sec:back}

In the context of a Festschrift for Ian Sloan, we presume that 
the reader is familiar with quasi-Monte Carlo, discrepancy 
and variation.  Modern approaches to QMC and discrepancy 
are covered in \cite{dick:pill:2010}. 
See \cite{variation} for an outline of variation for QMC including 
variation in the senses of Vitali and of Hardy and Krause. 

We will use a notion of functions that are singular but not
too badly singular.  
\begin{definition}\label{def:okdiagsing}
The function $f$ defined on $[0,1]^2$ has a diagonal singularity 
no worse than $|x_1-x_2|^{-A}$ for $0<A<1$, if
\begin{equation}\label{eq:regu}
\begin{split}
|f(\bsx)| & \le B|x_1-x_2|^{-A}\\
\biggl|\frac{\partial f(\bsx)}{\partial x_j}\biggr| & \le B|x_1-x_2|^{-A-1},\quad j\in\{1,2\},\quad\text{and}\\
\biggl|\frac{\partial^2 f(\bsx)}{\partial x_j\partial x_k}\biggr| & \le B|x_1-x_2|^{-A-2},\quad j,k\in\{1,2\}
\end{split}
\end{equation}
all hold for some $B<\infty$.
\end{definition}
We take $A>0$ in order to allow a singularity and $A<1$ because 
$f$ must be integrable. 
Smaller values of $A$ describe easier cases to handle. 
The value of $A$ to use for a given integrand may be evident from 
its analytical form.  If $A<1/2$ then $f^2$ is integrable. 
Definition~\ref{def:okdiagsing} is modeled on some previous notions:
\begin{definition}\label{def:oklowersing}
The function $f(\bsx)$ defined on $[0,1]^d$ has a lower edge singularity no worse 
than $\prod_{j=1}^d x_j^{-A_j}$, for constants $0<A_j<1$, if
$$
|\partial^u f(\bsx)| \le B\prod_{j=1}^d x_j^{-A_j-1_{j\in u}},
$$
holds for some $B<\infty$ and all $u\subseteq\{1,2,\dots,d\}$. 
\end{definition}
\begin{definition}
\label{def:okpointsing}
The function $f(\bsx)$ defined on $[0,1]^d$ has a point singularity no worse 
than $\Vert \bsx-\bsz\Vert^{-A}$, for $\bsz\in[0,1]^d$,  if 
$$
|\partial^u f(\bsx)| \le B\Vert\bsx-\bsz\Vert^{-A-|u|}
$$
holds for some $B<\infty$ and all $u\subseteq\{1,2,\dots,d\}$. 
\end{definition}
Definition~\ref{def:oklowersing} is one of several conditions in \cite{haltavoid} 
for singularities that arise as $\bsx$ approaches the boundary of the unit cube.
Definition~\ref{def:okpointsing} is used in \cite{owen2006quasi}
for isolated point singularities.
Definition~\ref{def:okdiagsing} is more stringent than 
Definitions~\ref{def:oklowersing} and \ref{def:okpointsing} are,
because it imposes a constraint on partial derivatives taken twice with respect
to $x_1$ or $x_2$.

To estimate $\mu =\int_{[0,1]^2} f(\bsx)\rd\bsx$ 
we will sample points $\bsx_i\in[0,1]^2$. 
The points we use will avoid a region near the singularity by sampling only within 
$$
S_\epsilon = \bigl\{ \bsx\in [0,1]^2\bigm| |x_1-x_2|\ge\epsilon\bigl\}
$$
where $0<\epsilon<1$. 
The set $S_\epsilon$ is the union of two disjoint triangles: 
\begin{align*}
T_\epsilon^u &= \bigl\{ \bsx\in[0,1]^2 \bigm| x_2\ge x_1+\epsilon\bigr\},\quad\text{and}\\
T_\epsilon^d &= \bigl\{ \bsx\in[0,1]^2 \bigm| x_2\le x_1-\epsilon\bigr\}.
\end{align*}
We let $-S_\epsilon$ denote the set $[0,1]^2\setminus S_\epsilon$. 
As remarked in the introduction, the set $T_u\cup T_d$
is not Sobol' extensible to $[0,1]^2$.

We will choose points $\bsx_{i,u}\in T_\epsilon^u$ for $i=1,\dots,n$ and estimate 
$\mu_{\epsilon,u}=\int_{T_\epsilon^u}f(\bsx)\rd\bsx$ by 
$$\hat\mu_{\epsilon,u}=\frac{{\mathrm{vol}}(T_\epsilon^u)}n \sum_{i=1}^nf(\bsx_{i,u}).$$
Using a similar estimate for $T_\epsilon^d$ we arrive at our estimate of $\mu$,
\begin{align*}
\hat\mu_{\epsilon} =\hat\mu_{\epsilon,u}+\hat\mu_{\epsilon,d}. 
\end{align*}
Our error then consists of two quadrature errors and a truncation 
error and it satisfies the bound 
\begin{align}\label{eq:theerr}
|\hat\mu_\epsilon -\mu| \le 
\biggl|\hat\mu_{\epsilon,u}-\int_{T_\epsilon^u}f(\bsx)\rd\bsx\biggr|
+\biggl|\hat\mu_{\epsilon,d}-\int_{T_\epsilon^d}f(\bsx)\rd\bsx\biggr|
+ \biggl|\int_{ 
-S_\epsilon}f(\bsx)\rd\bsx\biggr|. 
\end{align}

\section{Error bounds}\label{sec:bounds}

We show in Proposition~\ref{prop:trunc} below that 
the truncation error bound 
$|\int_{-S_\epsilon}f(\bsx)\rd\bsx|$ is $O(\epsilon^{1-A})$ as $\epsilon\to 0$. 
We will use the construction from~\cite{basu:owen:2015} and 
the Koksma-Hlawka inequality from \cite{bran:colz:giga:trav:2013}
to provide an upper bound for the 
integration error over $T_\epsilon^u$.  That bound grows as $\epsilon\to0$
and so to trade them off we will tune the way $\epsilon$ depends on~$n$. 

\begin{proposition}\label{prop:trunc}
Under the regularity conditions~\eqref{eq:regu},
$$
\biggl|\int_{-S_\epsilon} f(\bsx)\rd\bsx\biggr| \le  \frac{2B\epsilon^{1-A}}{1-A}. 
$$
\end{proposition}
\begin{proof}
We take the absolute value inside the integral and obtain 
\begin{align*}
\int_{-S_\epsilon} |f(\bsx)|\rd\bsx\le 
\int_{-S_\epsilon} B|x_1-x_2|^{-A}\rd\bsx 
 \le B\int_0^12\int_0^\epsilon x_2^{-A}\rd x_2\rd x_1 
\end{align*}
from which the conclusion follows. 
\end{proof}

Next we turn to the quadrature errors over $T_\epsilon^u$. Of course, $T_\epsilon^d$ is similar. 
The Koksma-Hlawka bound in 
\cite{bran:colz:giga:trav:2013}
 has 
$$
|\hat\mu_{\epsilon,u} - \mu_{\epsilon,u} | \le D_{T^u_\epsilon}^*(\bsx_{1,u},\dots,\bsx_{n,u}) V_{T^u_\epsilon}(f) 
$$
where $D_{T_{u,\epsilon}}^*$ and $V_{T^u_\epsilon}$ are measures of discrepancy and variation suited to the triangle. 
Basu and Owen \cite{basu:owen:2015}
provide a construction in which $D_{T_\epsilon^u}^* = O(\log(n)/n)$, the best possible rate. 

Brandolini et al.\ \cite[p. 46]{bran:colz:giga:trav:2013}
provide a bound for $V_{T_\epsilon^u}$, the variation on the simplex 
as specialized to the triangle. 
To translate their bound into our setting, we introduce 
the notation $f_{rs} = \partial^{r+s}f/\partial^rx_1\partial^sx_2$. 
Specializing their bound to the domain $T^u_\epsilon$ we find that the variation is 
{\allowdisplaybreaks
\begin{align}\label{eq:vorder} 
O\Bigl( & |f(0,1)| + |f(0,\epsilon)|+ |f(1-\epsilon,1)|\notag\\
&+\int_{\epsilon}^1|f(0,x_2)|\rd x_2 +\int_0^{1-\epsilon}|f(x_1,1)|\rd x_1 + \int_0^{1-\epsilon} |f(x_1,x_1+\epsilon)|\rd x_1\notag\\
&+\int_{\epsilon}^1|f_{01}(0,x_2)|\rd x_2 +\int_0^{1-\epsilon}|f_{10}(x_1,1)|\rd x_1 \\
& + \int_0^{1-\epsilon} |f_{10}(x_1,x_1+\epsilon)|\rd x_1 + \int_0^{1-\epsilon} |f_{01}(x_1,x_1+\epsilon)|\rd x_1\notag\\
& + \int_{T_\epsilon^u}|f(\bsx)|+|f_{01}(\bsx)|+|f_{10}(\bsx)|+|f_{20}(\bsx)|+|f_{02}(\bsx)|+|f_{11}(\bsx)| \rd\bsx\Bigr) \notag 
\end{align} 
as $\epsilon\to0$. The implied constant in~\eqref{eq:vorder} includes their unknown constant $C_2$, 
the reciprocals of  edge lengths 
of $T_{\epsilon}^u$, the reciprocal of the area of $T_\epsilon^u$, some small integers and some factors 
involving $\sqrt{2}(1-\epsilon)$, the length of the hypotenuse of $T_\epsilon^u$. 
}

\begin{proposition}\label{prop:quad}
Let $f$ satisfy the regularity condition~\eqref{eq:regu}. Then the trapezoidal variation of 
$f$ over $T_\epsilon^u$ satisfies 
$$
V_{T_\epsilon^u}(f) = O(\epsilon^{-A-1}) 
$$
as $\epsilon\to0$. 
\end{proposition}
\begin{proof}
Under condition~\eqref{eq:regu},
$$|f(0,1)|  
+ \int_\epsilon^1|f(0,x_2)|\rd x_2 
+ \int_0^{1-\epsilon}|f(x_1,1)|\rd x_1 
+\int_{T_\epsilon^u}|f(\bsx)|= O(1).$$
Next 
$$|f(0,\epsilon)| +|f(1-\epsilon,1)| +\int_0^{1-\epsilon}|f(x_1,x_1+\epsilon)|\rd x_1 
= O(\epsilon^{-A})$$ and 
$$\int_\epsilon^1|f_{01}(0,x_2)|\rd x_2 +\int_0^{1-\epsilon}|f_{10}(x_1,1)|\rd x_1 
= O(\epsilon^{-A})$$
as well. 
Continuing through the terms, we find that 
$$
\int_0^{1-\epsilon} |f_{10}(x_1,x_1+\epsilon)|\rd x_1 + \int_0^{1-\epsilon} |f_{01}(x_1,x_1+\epsilon)|\rd x_1 = O(\epsilon^{-A-1}). 
$$
The remaining terms are integrals of absolute partial derivatives of $f$ over $T_\epsilon^u$.  They 
are dominated by integrals of second derivatives and those terms obey the bound 
$$
\int_0^{1-\epsilon}\int_{x_1+\epsilon}^1B_2|x_1-x_2|^{-A-2}\rd x_2\rd x_1 = O(\epsilon^{-A-1}).
$$
\end{proof}

\begin{theorem}\label{thm:therate}
Under the regularity conditions~\eqref{eq:regu}, 
we may choose 
$\epsilon \propto \sqrt{\log(n)/n}$ and get 
\begin{align}\label{eq:rate}
|\hat\mu -\mu| = O\Bigl( \Bigl(\frac{\log(n)}{n}\Bigr)^{(1-A)/2}\Bigr). 
\end{align}
\end{theorem}
\begin{proof}
From Propositions \ref{prop:trunc} and~\ref{prop:quad}
we get 
$$|\hat \mu-\mu| = O 
\Bigl(\epsilon^{1-A} + \frac{\log(n)}{n}\epsilon^{-1-A}\Bigr).$$
Taking $\epsilon$ to be a positive multiple of $\sqrt{\log(n)/n}$ yields 
the result. 
\end{proof}

The choice of $\epsilon\propto \sqrt{\log(n)/n}$ optimizes the upper 
bound in~\eqref{eq:rate}. 

\section{Extension based approaches}\label{sec:alt}

Another approach to this problem is to construct a 
function $\tilde f$ where 
$\tilde f(\bsx) = f(\bsx)$ for $\bsx\in S_\epsilon$
and apply QMC to $\tilde f$. 
The function $\tilde f$ can smoothly bridge the 
gap between $T_\epsilon^u$ and $T_\epsilon^d$. 
With such a function, the quadrature error satisfies 
\begin{align}\label{eq:extbound}
\biggl| \frac1n\sum_{i=1}^n\tilde f(\bsx_i) - \int_{[0,1]^2}f(\bsx)\rd\bsx\biggr|
\le 
D_n^*(\bsx_1,\dots,\bsx_n)V_{{\mathrm HK}}(\tilde f) 
+ \int_{-S_\epsilon}|f(\bsx)-\tilde f(\bsx)|\rd\bsx 
\end{align}
where $V_{{\mathrm HK}}$ is total variation in the sense of Hardy and Krause. 

Our regularity condition~\eqref{eq:regu} allows for $f$
to take the value $\epsilon^{-A}$ along the line $x_2= x_1-\epsilon$
and to take the value $-\epsilon^{-A}$ along $x_2=x_1+\epsilon$. 
By placing squares of side $2\epsilon$ along the main diagonal 
we then find that the Vitali variation of an extension $\tilde f$ is at 
least $\lfloor (2\epsilon)^{-1}\rfloor 2\epsilon^{-A}\sim \epsilon^{-1-A}$. 
Therefore the Hardy-Krause variation of $\tilde f$ grows at least this quickly 
for some of the functions $f$ that satisfy~\eqref{eq:regu}. 
More generally, for singular functions along a linear manifold $M$
within $[0,1]^d$, and no worse than  $\mathrm{dist}(\bsx,M)^{-A}$,
an extension over the region within $\epsilon$ of $M$
could have a variation lower bound growing as fast as $\epsilon^{-(d-1)-A}$. 

This result is much less favorable than the one for isolated 
point singularities 
\cite{owen2006quasi}.  For integrands on $[0,1]^d$ no worse than $\Vert \bsx-\bsz\Vert^{-A}$,
where $\bsz\in[0,1]^d$,
Sobol's low variation extension yields a function $\tilde f$
that agrees with $f$ for $\Vert\bsx-\bsz\Vert\ge \epsilon>0$
having $V_{\mathrm HK}(\tilde f) = O(\epsilon^{-A})$. Here we see that no 
extension can have such low variation
for this type of singularity.

Owen \cite{haltavoid} considers functions with singularities 
along the lower boundary of $[0,1]^d$ that are no worse than $\prod_{j=1}^dx_j^{-A_j}$. 
Sobol's extension from the region where $\prod_jx_j\ge \epsilon$ has 
variation $O(\epsilon^{-\max A_j})$ when the $A_j$ are distinct (otherwise 
logarithmic factors enter).  So that problem with singularities along 
the boundary also has a more accurate 
extension than can be obtained for singularities along the diagonal. 

No extension $\tilde f$ from $S_\epsilon$ to $[0,1]^2$ can yield 
a bound~\eqref{eq:extbound} with a better rate 
than $O( (\log n/n)^{(1-A)/2})$.  To show this we first clarify 
one of the rules we impose on extensions.  When we extend
$f$ from $\bsx\in S$ to values of $\bsx\not\in S$ we do
not allow the construction of $\tilde f$ to depend on
$f(\bsx)$ for $\bsx\not\in S$.  That is, we cannot peek outside
the set we are extending from.  Some such rule must be necessary
or we could trivially get $0$ error from an extension based on
an oracle that uses the value of $\mu$ to define $\tilde f$.
With our rule, any two functions $f_1$ and $f_2$ with $f_1(\bsx)=f_2(\bsx)$
on $S_\epsilon$  have the same extension $\tilde f$.
From the triangle inequality,
$$\max_{j=1,2}\biggl(
\int_{-S_\epsilon} |\tilde f(\bsx)-f_j(\bsx)|\rd\bsx
\biggr) 
\ge \frac12\int_{-S_\epsilon}|f_1(\bsx)-f_2(\bsx)|\rd\bsx. 
$$
Now let 
\begin{align*}
f_1(\bsx) &= \begin{cases}
 -|x_1-x_2|^{-A}, & x_2-x_1>0\\
\phantom{-} |x_1-x_2|^{-A}, & x_2-x_1<0, 
\end{cases}
\end{align*}
and
\begin{align*}
f_2(\bsx) &= \begin{cases} 
|x_1-x_2|^{-A}, & \phantom{-\epsilon > }\,\,x_2-x_1>0\\
\phi(x_2-x_1), &  \phantom{-}0 > x_2-x_1 \ge -\epsilon\\
|x_1-x_2|^{-A}, & -\epsilon\,  >  x_2 -x_1, 
\end{cases}
\end{align*}
for a quadratic polynomial $\phi$
with $\phi(-\epsilon)=\epsilon^{-A}$,
 $\phi'(-\epsilon)=-A\epsilon^{-A-1}$,
and  $\phi''(-\epsilon)=A(A+1)\epsilon^{-A-2}$. 
Both $f_1$ and $f_2$ satisfy~\eqref{eq:regu}
and $\int_{-S_\epsilon} |f_1(\bsx)-f_2(\bsx)|\rd\bsx$ is 
larger than a constant times $\epsilon^{1-A}$. 
That is the same rate as the truncation error 
from Proposition~\ref{prop:trunc} and the 
quadrature error from this approach also 
attains the same rate as the error in Proposition~\ref{prop:quad}. 
As a result, we conclude that even if we could construct 
the best extension $\tilde f$, it would not lead to a bound 
with a better rate than the one in Theorem~\ref{thm:therate}.

\section{Transformation}\label{sec:trans}

Here we consider applying a change of variable
to move the singularity from the diagonal to an
edge of the unit square.
We focus on integrating
$f(\bsx)$ over 
$T^u=\{ (x_1,x_2)\in[0,1]^2\mid 0\le x_1\le x_2\le 1\}$
for $f$ with a singularity no worse than $|x_1-x_2|^{-A}$.
The same strategy and same convergence rate hold on
$T^d=\{ (x_1,x_2)\in[0,1]^2\mid 0\le x_2\le x_1\le 1\}$.
Using a standard change of variable we have
$$
\int_{T^u}f(\bsx)\rd\bsx 
= \frac12\int_0^1\int_0^1 f((1-u_1)\sqrt{u_2},\sqrt{u_2})\rd\bsu,
$$
which we then write as
$$
\frac12\int_{[0,1]^2}g(\bsu)\rd\bsu,\quad\text{for $g(\bsu)=f((1-u_1)\sqrt{u_2},\sqrt{u_2})$}.
$$
That is $g(\bsu) = f(\tau(\bsu))$ for a transformation $\tau:[0,1]^2\to T_u\subset[0,1]^2$ 
given by $\tau_1(\bsu)=(1-u_1)\sqrt{u_2}$ and $\tau_2(\bsu)=\sqrt{u_2}$.

The archetypal function with diagonal singularity satisfying
Definition~\ref{def:okdiagsing} is $f(\bsx) = |x_1-x_2|^{-A}$.
The corresponding function $g$ for this $f$ is
$$g(\bsu) = |\tau_1(\bsu)-\tau_2(\bsu)|^{-A}
= u_1^{-A}u_2^{-A/2}.
$$
We see that the change of variable has produced an integrand with a singularity
no worse than $u_1^{-A}u_2^{-A/2}$ according to  Definition~\ref{def:oklowersing}.
Taking $\bsu_i$
to be the Halton points leads to a quadrature 
error at rate $O(n^{-1+\epsilon+A})$ for any $\epsilon>0$,
because Halton points 
(after the zeroth one) avoid the origin at a suitable rate 
\cite [Corollary 5.6]{haltavoid}. 
For this integrand $g$, randomized quasi-Monte Carlo points for will attain 
the mean error rate 
$\bbE( |\hat \mu-\mu|) =O(n^{-1+\epsilon+A})$
as shown in Theorem 5.7 of \cite{haltavoid}.

We initially thought that the 
conversion from a diagonal singularity to a lower edge
singularity no worse than $u_1^{-A}u_2^{-A/2}$ 
would follow for other functions satisfying Definition~\ref{def:okdiagsing}.
Unfortunately, that is not necessarily the case.

Let $f$ be defined on $[0,1]^2$ with a diagonal singularity no
worse than $|x_1-x_2|^{-A}$ for $0<A<1$.
First,
$$
|g(\bsu)|
= |f((1-u_1)\sqrt{u_2},\sqrt{u_2})|
 \le B|u_1u_2^{1/2}|^{-A}
$$
which fits Definition~\ref{def:oklowersing}. Similarly,
\begin{align*}
g_{10}(\bsu)&
=f_{10}(\tau_1(\bsu),\tau_2(\bsu))\frac{\partial\tau_1(\bsu)}{\partial u_1}
=O(|\tau_1-\tau_2|^{-A-1})\times u_2^{1/2}=O(u_1^{-A-1}u_2^{-A/2}) 
\end{align*}
which also fits Definition~\ref{def:oklowersing}. However,
\begin{align}
g_{01}(\bsu)&
=f_{10}(\tau(\bsu))\frac{\partial\tau_1(\bsu)}{\partial u_2}
+f_{01}(\tau(\bsu))\frac{\partial\tau_2(\bsu)}{\partial u_2}\notag\\
&
=\bigl(f_{10}(\tau(\bsu))+f_{01}(\tau(\bsu))\bigr)
\frac12u_2^{-1/2}
-f_{10}(\tau(\bsu))\frac12u_1u_2^{-1/2}.\label{eq:g01parts}
\end{align}
Now $f_{10}$ and $f_{01}$ appearing in~\eqref{eq:g01parts}  are
both $O(u_1^{-A-1}u_2^{-A/2-1/2})$.  Therefore the two terms there are
$O(u_1^{-A-1}u_2^{-A/2-1})$ and $O(u_1^{-A}u_2^{-A/2-1})$ respectively.
The first term is too large by a factor of $u_1^{-1}$ to suit Definition~\ref{def:oklowersing}.
We would need $(f_{01}+f_{10})(\tau(\bsu))$ to be only $O( u_1^{-A}u_2^{-A/2-1/2})$.
Definition~\ref{def:okdiagsing} is also not strong enough for $g_{11}$
to be $O( u_1^{-A-1}u_2^{-A/2-1})$ as it would need to be under Definition~\ref{def:oklowersing}.  
That definition
yields only $O(u_1^{-A-2}u_2^{-A/2-1})$ without stronger assumptions.
Theorem~\ref{thm:modulated} below gives a sufficient condition
where $f$ is a modulated version of $|x_1-x_2|^{-A}$.

\begin{theorem}\label{thm:modulated}
Let $f(\bsx) = |x_1-x_2|^{-A}h(\bsx)$ for $\bsx\in[0,1]^2$ and $0<A<1$
where $h$ and its first two derivatives are bounded.
Then $g(\bsu) = f((1-u_1)\sqrt{u_2},\sqrt{u_2})$ satisfies
Definition~\ref{def:oklowersing} with $A_1=A$ and $A_2=A/2$.
\end{theorem}
\begin{proof}
We begin with
$$
g(\bsu) = u_1^{-A}u_2^{-A/2}h((1-u_1)u_2^{1/2},u_2^{1/2}) = O(u_1^{-1}u_2^{-A/2})
$$
by boundedness of $h$.
Next because $u_1$ is not in  the second argument to $h$,
\begin{align*} 
g_{10}(\bsu) 
&= -Au_1^{-A-1}u_2^{-A/2}h(\tau(\bsu))  
+u_1^{-A}u_2^{-A/2}  h_{10}(\tau(\bsu)) \partial \tau_1(\bsu)/\partial u_1\\
&= -Au_1^{-A-1}u_2^{-A/2}h(\tau(\bsu))  
   -u_1^{-A}u_2^{-A/2+1/2}h_{10}(\tau(\bsu))\\
& = O(u_1^{-A-1}u_2^{-A/2})  
\end{align*} 
as required.
Similarly,
\begin{align*}
g_{01}(\bsu) &= 
-(A/2)u_1^{-A}u_2^{-A/2-1}h(\tau(\bsu)) 
\\&\quad
+u_1^{-A}u_2^{-A/2}
\bigl( h_{10}(\tau(\bsu))(1-u_1)+h_{01}(\tau(\bsu))
\bigr) (1/2)u_2^{-1/2}\\
&= O(u_1^{-A}u_2^{-A/2-1})
\end{align*}
as required. Finally $g_{11}(\bsu)$ equals
\begin{align*}
&(A^2/2)u_1^{-A-1}u_2^{-A/2-1}h(\tau(\bsu))\\
&\quad -(A/2)u_1^{-A}u_2^{-A/2-1}h_{10}(\tau(\bsu))(-u_2^{1/2})\\
&\quad -(A/2)u_1^{-A-1}u_2^{-A/2-1/2}\bigl(h_{10}(\tau(\bsu))(1-u_1)+h_{01}(\tau(\bsu))\bigr)\\
&\quad +(u_1^{-A}u_2^{-A/2-1/2}/2)\bigl(
-h_{10}(\tau(\bsu)) +(1-u_1)h_{20}(\tau(\bsu))(-u_2^{1/2})
+h_{11}(\tau(\bsu))(-u_2^{1/2})\bigr)\\
&=O(u_1^{-A-1}u_2^{-A/2-1})
\end{align*}
as required.
\end{proof}

\section{Discussion}\label{sec:disc}

We find that for an integrand with a singularity 
`no worse than $|x_1-x_2|^{-A}$' along the line $x_1=x_2$
we can get a QMC estimate with error $O( (\log(n)/n)^{(1-A)/2})$
by splitting the square into two triangles and ignoring
a region in between them.
The same method applies to singularities along 
the other diagonal of $[0,1]^2$. Moreover,
the result extends to singularities along other lines 
intersecting the square.  One can partition the square 
into rectangles, of which one has the singularity along 
the diagonal while the others have no singularity, and 
then integrate $f$ over each of those rectangles. 

That result does not directly extend to 
singularities along a linear manifold in $[0,1]^d$
for $d\ge 3$. The reason is that the QMC result for 
integration in the triangle from \cite{basu:owen:2015} has 
not been extended to the simplex. In a personal communcation,
Dimitry Bilyk told us that such an extension would imply a 
counterexample to the Littlewood conjecture, which is widely 
believed to be true. 
Basu and Owen \cite{basu2015scrambled} present some algorithms for 
RQMC over simplices, but they come without a Koksma-Hlawka 
bound that would be required for limiting arguments using 
sequences of simplices.

The rate $O( (\log(n)/n)^{(1-A)/2})$ is a bit disappointing.  
We do much better by transforming the problem to place
the singularity along the boundary of a square region,
for then we can attain $O(n^{-1+\epsilon+A})$, under a stronger
assumption that $f$ is our prototypical singular function $|x_1-x_2|^{-A}$
possibly modulated by a function $h$ with bounded second derivatives on $[0,1]^2$.
As a result we find that there is something to be gained by
engineering QMC-friendly singularities in much the same way
that benefits of QMC-friendly discontinuities have been found
valuable by Wang and Sloan \cite{wang2011quasi}.

\begin{acknowledgement}
This work was supported by the US National Science Foundation under
grants DMS-1407397 and  DMS-1521145.
We thank two anonymous reviewers for helpful comments.
\end{acknowledgement}

%
%

\end{document}